\documentclass{amsart}
\usepackage{amsmath}

\usepackage{amsthm}

\usepackage{setspace}
\usepackage[dvips]{graphicx}

\usepackage{psfrag}

\usepackage{amssymb} 
\def\M{{\mathbb{M}}}
\def\N{{\mathbb{N}}}
\def\S{{\mathbb{S}}}

\theoremstyle{plain}
\newtheorem{thm}{Theorem}
\newtheorem{theorem}[thm]{Theorem}
\newtheorem{lemma}[thm]{Lemma}
\newtheorem{proposition}[thm]{Proposition}

\theoremstyle{remark}
\newtheorem{case}{Case}
\newtheorem{keyword}{keyword}
\newtheorem{conjecture}{Conjecture}

\theoremstyle{definition}

\begin{document}

\title[Local structures in polyhedral maps on surfaces, and path transferability of graphs]{Local structures in polyhedral maps on surfaces, and path transferability of graphs}

\author{Ryuzo Torii}

\address{Department of Mathematics, School of Education,
Waseda University, Nishi-waseda 1-6-1, Shin'juku-ku Tokyo 169-8050, Japan}

\email{torii@toki.waseda.jp}

\date{\today}

\maketitle

\allowdisplaybreaks

\begin{abstract}
We extend Jendrol' and Skupie\'n's results about the local structure of maps on the $2$-sphere: In this paper we show that if a polyhedral map $G$ on a surface $\M$ of Euler characteristic $\chi (\M) \le 0$ has more than $126|\chi (\M)|$ vertices, then $G$ has a vertex with "nearly" non-negative combinatorial curvature. As a corollary of this, we can deduce that path transferability of such graphs are at most $12$.
\end{abstract}
\begin{keyword}
Polyhedral maps, Embedding, Light vertex, Combinatorial curvature, Path transferability
\end{keyword}
\section{Introduction}
In this paper we use standard terminology and notation of graph theory. The graphs discussed here are finite, simple, undirected, and connected. A orientable surface $\S \sb{g}$ of genus $g$ is obtained from the sphere by adding $g$ handles. A non-orientable surface $\N \sb{q}$ of genus $q$ is obtained from the sphere by adding $q$ crosscaps. The Euler characteristic is defined by
\begin{eqnarray*}
\chi (\S \sb{g}) = 2-2g, \hspace{9mm}  \chi (\N \sb{q}) = 2-q.
\end{eqnarray*}

If a graph $G$ is embedded in a surface $\M$ then the connected components of $\M - G$ are called the {\it faces} of $G$.
If each face is an open disc then the embedding is called a {\it $2$-cell embedding}.
If $G$ is a $2$-cell embedding in a surface $\M$ and each vertex has degree at least three, then $G$ is called a {\it map} on $\M$. If in addition, $G$ is $3$-connected and the embedding has representativity at least three, then $G$ is called {\it polyhedral} map in $\M$ (see e.g. \cite{RoVi}). Let us recall that the {\it representativity} (or face-width) of a $2$-cell embedded graph $G$ in a surface $\M$ is equal to the smallest number $k$ such that $\M$ contains a non-contractible closed curve that intersects the graph $G$ in $k$ points.

The {\it facial walk} of a face $\alpha $ in a $2$-cell embedding is the shortest closed walk that follows the edges in order around the boundary of the face $\alpha $. The {\it degree} of a face $\alpha $ is the length of its facial walk. Vertices and faces of degree $i$ are called {\it $i$-vertices} and {\it $i$-faces}, respectively.
A vertex $v$ is said to be an {\it $(a\sb {1}, a\sb {2}, \ldots , a\sb {n})$-vertex} if the faces incident with $v$ have degree $a\sb {1}, a\sb {2}, \ldots , a\sb {n}$.
An edge $e$ is said to be an {\it $(i, j)$-edge} if two vertices incident with $e$ have degree $i, j$.

If each facial walk of a $2$-cell embedding consists of distinct vertices, then the embedding is called {\it closed $2$-cell embedding}. If $G$ is a closed $2$-cell embedding and the subgraph of $G$ bounding the faces incident with any vertex is a wheel with $\ge 3$ spokes and a possibly subdivided rim, the embedding is called a {\it wheel-neighborhood embedding}. The following Proposition is due to Negami and Vitray (see \cite{MoTh}, \cite{Ne}, \cite{RoVi}, \cite{Vi}).
\begin{proposition}
An embedding of a graph is a polyhedral map if and only if it is a wheel-neighborhood embedding.
\label{th:a}
\end{proposition}

By Euler polyhedral formula, a simple planar graph has a vertex of degree $\le 5$. Local structures of planar graphs are further studied by Jendrol' and Skupie\'n\cite{JeSk}. Local structures of graphs on general surfaces is investigated by several researchers, see \cite{JeTuVo}, \cite{JeVo}, \cite{Ko}, \cite{Tu} etc. In this paper we will show the following by using the Discharging method that is slightly changed from the ones of \cite {JeSk}, \cite {JeTuVo}:

\begin{theorem}
Let $G$ be a simple polyhedral map on a surface $\M$ of Euler characteristic $\chi (\M) \le 0$. If $G$ has more than $126|\chi (\M)|$ vertices, then $G$ contains an $(a\sb {1}, a\sb {2}, \ldots , a\sb {n})$-vertex, where $n=3,4,5,6$ and $(a\sb {1}, a\sb {2}, \ldots , a\sb {n})$ satisfies one of the lists in Table \ref{tab:1}:
\end{theorem}

\begin{table}
\label{tab:1}
\begin{tabular}{lllll}
\hline\noalign{\smallskip}
(1) In the case $n=3$ 	 		 & & (2) In the case $n=4$ 			 & &   \\
$(3,3,\forall )^{\dagger}$		 & & $(3,3,3,\forall )^{\dagger}$ 	 & &   \\
$(3,4,\forall )^{\dagger}$		 & & $(3,3,4,\le 2518 )$ 			 & &   \\
$(3,5,\forall )^{\dagger}$		 & & $(3,3,5,\le 2518 )$ 			 & &   \\
$(3,6,\le 2518 )$				 & & $(3,3,6,6)^{\dagger}$ 			 & &   \\
$(3,7,\le 2518 )$ 				 & & $(3,4,4,4)^{\dagger}$ 			 & &   \\
$(3,8,\le 2518 )$ 				 & & $(3,4,4,5)^{\dagger}$ 			 & &   \\
$(3,9,\le 2518 )$			 	 & & $(3,4,4,6)^{\dagger}$ 			 & &   \\
$(3,10,\le 2518 )$		 		 & & $(4,4,4,4)^{\dagger}$ 			 & &   \\
$(3,11,\le 2518 )$ 				 & & 					 			 & &   \\
$(3,12,\le 2518 )$ 				 & & (3) In the case $n=5$			 & &   \\
$(4,4,\forall )^{\dagger}$ 		 & & $(3,3,3,3,3)^{\dagger}$ 		 & &   \\
$(4,5,\le 2518 )$ 				 & & $(3,3,3,3,4)^{\dagger}$ 		 & &   \\
$(4,6,\le 2518 )$ 				 & & $(3,3,3,3,5)^{\dagger}$ 		 & &   \\
$(4,7,\le 2518 )$ 				 & & $(3,3,3,3,6)^{\dagger}$ 		 & &   \\
$(4,8,\le 2518 )$ 				 & & $(3,3,3,4,4)^{\dagger}$ 		 & &   \\
$(5,5,\le 2518 )$ 				 & & 								 & &   \\
$(5,6,\le 2518 )$ 				 & & (4) In the case $n=6$			 & &   \\
$(6,6,6 )^{\dagger}$ 			 & & $(3,3,3,3,3,3)^{\dagger}$		 & &   \\
\noalign{\smallskip}\hline 
\noalign{\smallskip}
\noalign{\smallskip}
\end{tabular}
\caption{Local structure in polyhedral map with sufficient large vertices}
\end{table}

Such $i$-vertices, $i=3,4,5,6$, are called {\it light} vertices in $G$. After Section $4$, these vertices are also called the ones with nearly non-negative curvature.\\

On the other hand, path transferability is introduced in \cite{To1}: We consider a path as an ordered sequence of distinct vertices with a head and a tail. Given a path, a {\it transfer-move} is to remove the tail and add a vertex at the head. A graph is {\it $n$-path-transferable} if any path with length $n$ can be transformed into any other such path by a sequence of transfer-moves. The maximum number $n$ for which $G$ is $n$-path-transferable is called the {\it path transferability} of $G$. The author in \cite{To2} showed the following result for planar graphs.
\begin{theorem}[\cite{To2}]
Path transferability of a simple planar graph with minimum degree $\ge 3$ is at most $10$.
\end{theorem}

For graphs on general surfaces, we will show the following result.\\
\\
\\
\textbf{Main Theorem.}
If a polyhedral map $G$ on a surface $\M$ of Euler characteristic $\chi (\M) \le 0$ has more than $126|\chi (\M)|$ vertices, then path transferability of $G$ is at most $12$.

\section{Proof of Theorem $2$}
Let $G$ be a counterexample with $n$ vertices. We consider only $2$-cell embeddings of graphs. Hence Euler's formula implies:
\begin{eqnarray*}
\sum_{v\in V(G)}^{}(2\text{deg}(v)-6) + \sum_{\alpha \in F(G)}^{}(\text{deg}(\alpha )-6) = 6|\chi (\M)|
\end{eqnarray*}
because $\chi (\M) \le 0$.

In the following we use the Discharging method. We assign to each vertex $v$ the charge $c(v)=2$deg$(v)-6$ and to each face $\alpha $ the charge $c(\alpha )=$ deg$(\alpha )-6$. These charges of the vertices and faces will be locally redistributed to charges $c\sp {*}(v)$ and $c\sp {*}(\alpha )$, respectively, by the following rules. We first apply Rule \textbf{A$1-4$}, and next apply Rule \textbf{B}. An edge $e$ is called {\it weak} or {\it semi-weak} if two or exactly one of its endvertices are of degree $3$, respectively. A face of $G$ is called {\it minor} if its degree is at most $5$, and is called {\it major} if its degree is at least $7$.\\
\\
\textbf{Rule A$1$}.\\
Suppose that $\alpha $ is a face of $G$ incident with a vertex $v$, and that deg$(\alpha ) \le 5 $, deg$(v) \ge 4$. Then $v$ sends to $\alpha $ the following charge:

$1$				if deg $(\alpha ) = 3$,

$\frac{1}{2}$	if deg $(\alpha ) = 4$,

$\frac{1}{5}$	if deg $(\alpha ) = 5$.
\\
\textbf{Rule A$2$}.\\
Every $k$-vertex, $k\ge 4$, which has at least one $3$-face and at least two $6$-faces sends additional charge $\frac{1}{10}$ to each $3$-face after Rule A$1$.
\\
\textbf{Rule A$3$}.\\
Suppose that $e$ is a common edge of adjacent faces $\alpha $ and $\alpha \sp {\prime }$ of $G$, and that deg$(\alpha ) \le 5 $, deg$(\alpha \sp {\prime }) \ge 7$. If $e$ is an weak edge, then $\alpha \sp {\prime }$ sends to $\alpha $ the following charge:\\

\begin{tabular}{rl||c|c|c|}
		 &	  		&deg$(\alpha )=3$   & $4$ 	& $5$	\\
\hline\noalign{\smallskip}
deg$(\alpha \sp {\prime })$& $=7,8$& $\frac{1}{5}$ & $\frac{1}{5}$ & $\frac{1}{5}$ \\
				 &$=9,10,11,12$ & $\frac{1}{2}$	& $\frac{1}{2}$	& $\frac{1}{5}$ \\
				 &$=13,\ldots ,2518$	   & $1$    & $\frac{1}{2}$	& $\frac{1}{5}$ \\
				 &$\ge 2519$	   & $\frac{19}{10}$    & $1$	& $\frac{2}{5}$ \\
\noalign{\smallskip}\hline 
\end{tabular}
\\
\\
If $e$ is a semi-weak edge, then $\alpha \sp {\prime }$ sends to $\alpha $ the following charge.\\

\begin{tabular}{rl||c|c|c|}
		 &	  		&deg$(\alpha )=3$   & $4$ 	& $5$	\\
\hline\noalign{\smallskip}
deg$(\alpha \sp {\prime })$& $=7,8$& $\frac{1}{10}$ & $\frac{1}{10}$ & $\frac{1}{10}$ \\
				 &$=9,10,11,12$ & $\frac{1}{4}$	& $\frac{1}{4}$	& $\frac{1}{10}$ \\
				 &$=13,\ldots ,2518$	   & $\frac{1}{2}$    & $\frac{1}{4}$	& $\frac{1}{10}$ \\
				 &$\ge 2519$	   & $1$    & $\frac{1}{2}$	& $\frac{1}{5}$ \\
\noalign{\smallskip}\hline 
\end{tabular}
\\
\\
\textbf{Rule A$4$}.\\
Every $k$-face, $k\ge 2519$, supplies the charge $\frac{1}{2}$, $\frac{1}{5}$  to each incident $(3,3,4,k)$-, $(3,3,5,k)$-vertices, respectively.\\

\begin{lemma}
After applying Rule \textbf{A$1-4$} to each vertex $v$ and each face $\alpha $, the new charge $c\sp {*}(v)$ and $c\sp {*}(\alpha )$ are 
\begin{eqnarray*}
c\sp {*}(v) \ge 0,
\end{eqnarray*}
\begin{eqnarray*}
c\sp {*}(\alpha ) \ge 
\left\{ 
\begin{array}{llll}
0 & \text{ if } & d=3,4,5,6,\\
\frac{2}{5}, \frac{6}{5}, 1, \frac{3}{2}, \frac{5}{2}, 3 & \text{ if } & d=7,8,9,10,11,12,\\
\frac{1}{21}d & \text{ if } & d\ge 13,
\end{array}
\right. 
\end{eqnarray*}
here $d= {\rm deg} (\alpha )$.
\label{th:b}
\end{lemma}

\begin{proof}
We consider several cases.
\begin{case} Let $v$ be a $k$-vertex, $k \ge 3$. Then $c(v)=2k-6$.\\

{\it 1-1}: $k=3$. Then $c\sp {*}(v) = c(v) =0$.\\

{\it 1-2}: $k=4$. This $4$-vertex $v$ corresponds to one of the following types: $(3,3,4, \ge 2519)$, $(3,3,5, \ge 2519)$, $(3,3,\ge 6, \ge 7)$, $(3,4,4, \ge 7)$, $(3,\ge 4, \ge 5, \ge 5)$, $(\ge 4,\ge 4,\ge 4,\ge 5)$ since $G$ is a counterexample. For the first type $(3,3,4, \ge 2519)$, $c\sp {*}(v) = c(v)-( 1\times 2 + \frac{1}{2}) + \frac{1}{2}  =0$ by \textbf{Rule A$4$}. For the second type $c\sp {*}(v)=0$ similarly. If this vertex has the type $(3,6,6,i)$, $i\ge 4$, then $c\sp {*}(v) \ge c(v)- (1+\frac{1}{2} + \frac{1}{10}) = \frac{2}{5} \ge 0$ by \textbf{Rule A$2$}. For the other cases we can deduce that $c\sp {*}(v) \ge 0$.\\

{\it 1-3}: $k=5$. Then $v$ corresponds to one of the following types: $(3,3,3,3, \ge 7)$, $(3,3,3,\ge 4,\ge 5)$, $(\ge 3,\ge 3,\ge 4,\ge 4,\ge 4)$. For the first type $c\sp {*}(v) = c(v)-1\times 4  =0$, and for the second type $c\sp {*}(v) \ge c(v) -(1+1+1+\frac{1}{2}+ \frac{1}{5} ) =\frac{3}{10} \ge 0$. For the third type $c\sp {*}(v) = c(v)-(1+1+\frac{1}{2}+\frac{1}{2}+\frac{1}{2} ) =\frac{1}{2} \ge 0$.

{\it 1-4}: $k=6$. Then $v$ has the type $(\ge 3,\ge 3,\ge 3,\ge 3,\ge 3,\ge 4)$, and $c\sp {*}(v) \ge c(v)- (1\times 5 + \frac{1}{2}) = \frac{1}{2}  \ge 0$.\\

{\it 1-5}: $k\ge 7$. The charge transferred from $v$ is maximized when all incident faces of $v$ are of degree $3$ or when all incident faces except two $6$-faces are of degree $3$. Therefore $c\sp {*}(v) \ge c(v)- \max \{ 1\times k;  (1+\frac{1}{10} )\times (k-2) \} = \min \{ k-6;  \frac{1}{10}(9k-38) \} \ge 0$.
\end{case}

\begin{case} Let $\alpha $ be a $k$-face, $k \ge 3$. Then $c(\alpha )=k-6$.

{\it 2-1}: $k=3$. Let $x\sb {1}, x\sb {2}, x\sb {3}$ be the three vertices of $\alpha $, and $\beta \sb {1}, \beta \sb {2}, \beta \sb {3}$ the three adjacent faces such that they have $x\sb {1}x\sb {2}$, $x\sb {2}x\sb {3}$, $x\sb {3}x\sb {1}$ in common with $\alpha $, respectively. 

[a] We first assume that all of $x\sb {1}, x\sb {2}, x\sb {3}$ are of degree $3$. One of the followings holds; [1] all of $\beta \sb {1}, \beta \sb {2}, \beta \sb {3}$ are of degree $\ge 13$; or [2] two of $\beta \sb {1}, \beta \sb {2}, \beta \sb {3}$ are of degree $\ge 2519$. Anyway $c\sp {*}(\alpha ) \ge c(\alpha ) + \min \{ 1\times 3 ; \frac{19}{10}\times 2  \}  \ge 0$.

[b] We next assume that exactly one of $x\sb {1}, x\sb {2}, x\sb {3}$ has degree $\ge 4$. Without loss of generality, let deg$(x\sb {1} )\ge 4$ and deg$(x\sb {2} )=$ deg$(x\sb {3} )= 3$. We further assume that deg$(\beta \sb {1} )=6$. Then deg$(\beta \sb {2} )\ge 2519$ because $x\sb {2}$ is not $(3,6, \le 2518)$-vertex. If deg$(\beta \sb {3} )=6$, then $c\sp {*}(\alpha ) = c(\alpha ) +( 1 +\frac{1}{10}) + \frac{19}{10}=0$ by \textbf{Rule A$2$}, therefore we set deg$(\beta \sb {3} )\ge 7$. The face $\beta \sb {3}$ sends to $\alpha $ the charge $\ge \frac{1}{10}$, hence $c\sp {*}(\alpha ) \ge c(\alpha ) +1 + ( \frac{19}{10} +\frac{1}{10}) =0$. We thus assume that deg$(\beta \sb {1} )\ge 7$. We similarly deduce that deg$(\beta \sb {3} )\ge 7$ by its symmetry. If one of $\beta \sb {1}, \beta \sb {3}$ has degree $\le 12$, then deg$(\beta \sb {2} )\ge 2519$, and then $c\sp {*}(\alpha ) \ge c(\alpha ) + 1 + ( \frac{1}{10} \times 2 + \frac{19}{10} ) \ge 0$. Therefore both $\beta \sb {1}$ and $\beta \sb {3}$ are of degree $\ge 13$. If $\beta \sb {2}$ has degree $\le 12$, then $\beta \sb {1}$ and $\beta \sb {3}$ are of degree $\ge 2519$, and then $c\sp {*}(\alpha ) \ge c(\alpha ) + 1 + 1\times 2 = 0$. Hence $\beta \sb {2}$ has degree $\ge 13$, and then $c\sp {*}(\alpha ) \ge c(\alpha ) + 1 + ( \frac{1}{2} \times 2 +1 ) = 0$. 

[c] We assume that exactly two of $x\sb {1}, x\sb {2}, x\sb {3}$ have degree $\ge 4$. Let deg$(x\sb {1} )\ge 4$, deg$(x\sb {2} )\ge 4$, and deg$(x\sb {3} )=3$. If one of the face $\beta \sb {2}$, $\beta \sb {3}$ has degree $\le 12$, then the other face has degree $\ge 2519$, and then $c\sp {*}(\alpha ) \ge c(\alpha ) + 1 \times 2 + 1 = 0$. Therefore both $\beta \sb {2}$ and $\beta \sb {3}$ are of degree $\ge 13$, and $c\sp {*}(\alpha ) \ge c(\alpha ) + 1 \times 2 + (\frac{1}{2} \times 2 ) = 0$.

[d] We finally assume that $x\sb {1}, x\sb {2}, x\sb {3}$ have degree $\ge 4$. Then $c\sp {*}(\alpha ) \ge c(\alpha ) + 1 \times 3 = 0$.\\

{\it 2-2}: $k=4$. Let $x\sb {1}, x\sb {2}, x\sb {3}, x\sb {4}$ be the four vertices of $\alpha $ in a natural circular ordering, and $\beta \sb {1}, \beta \sb {2}, \beta \sb {3},\beta \sb {4}$ the adjacent faces such that they have $x\sb {1}x\sb {2}$, $x\sb {2}x\sb {3}$, $x\sb {3}x\sb {4}$, $x\sb {4}x\sb {1}$ in common with $\alpha $, respectively.

[a] We assume that all of $x\sb {1}, x\sb {2}, x\sb {3}, x\sb {4}$ are of degree $3$. Then [1] all of $\beta \sb {1}, \beta \sb {2}, \beta \sb {3}, \beta \sb {4}$ are of degree $\ge 9$; or [2] at least two of the four faces are of degree $\ge 2519$. Therefore $c\sp {*}(\alpha ) \ge c(\alpha ) + \min \{ 1\times 2 ; \frac{1}{2} \times 4 \} = 0$.

[b] We next assume that exactly one of $x\sb {1}, x\sb {2}, x\sb {3}, x\sb {4}$ are of degree $\ge 4$. Let deg$(x\sb {1} )\ge 4$ and deg$(x\sb {2} )=$ deg$(x\sb {3} )=$ deg$(x\sb {4} )= 3$. If deg$(\beta \sb {2} )\le 8$, then $\beta \sb {1} $ and $\beta \sb {3} $ are of degree $\ge 2519$, and then $c\sp {*}(\alpha ) \ge c(\alpha ) + \frac{1}{2} + (\frac{1}{2} + 1 )= 0$. Therefore deg$(\beta \sb {2} )\ge 9$, and similarly deg$(\beta \sb {3} )\ge 9$ by its symmetry. If deg$(\beta \sb {1} )\le 8$, then $\beta \sb {2} $ has degree $\ge 2519$, and then $c\sp {*}(\alpha ) \ge c(\alpha ) + \frac{1}{2} + (1+ \frac{1}{2}  )= 0$. We thus conclude that deg$(\beta \sb {1} )\ge 9$, and that deg$(\beta \sb {4} )\ge 9$ in the same way. Hence all of $\beta \sb {1}, \beta \sb {2}, \beta \sb {3},\beta \sb {4}$ have degree $\ge 9$, and $c\sp {*}(\alpha ) \ge c(\alpha ) + \frac{1}{2} + (\frac{1}{2}\times 2  + \frac{1}{4}\times 2 )= 0$.

[c] We assume that consecutive two vertices of $x\sb {1}, x\sb {2}, x\sb {3}, x\sb {4}$ are of degree $\ge 4$. Let deg$(x\sb {1} )\ge 4$, deg$(x\sb {2} )\ge 4$, and deg$(x\sb {3} )=$ deg$(x\sb {4} )= 3$. If deg$(\beta \sb {3} )\le 8$, then $\beta \sb {2} $ and $\beta \sb {4} $ are of degree $\ge 2519$, and then $c\sp {*}(\alpha ) \ge c(\alpha ) + ( \frac{1}{2} \times 2 ) + (\frac{1}{2} \times 2)= 0$. Therefore deg$(\beta \sb {3} )\ge 9$. If deg$(\beta \sb {2} )\le 8$, then $\beta \sb {3} $ has degree $\ge 2519$, and then $c\sp {*}(\alpha ) \ge c(\alpha ) + ( \frac{1}{2} \times 2 ) + 1= 0$. Hence deg$(\beta \sb {2})\ge 9$, and similarly deg$(\beta \sb {4})\ge 9$. Then $c\sp {*}(\alpha ) \ge c(\alpha ) + ( \frac{1}{2} \times 2) + (\frac{1}{2}  + \frac{1}{4}\times 2 )= 0$.

[d] We assume that opposite two vertices of $x\sb {1}, x\sb {2}, x\sb {3}, x\sb {4}$ are of degree $\ge 4$. Let deg$(x\sb {1} )\ge 4$, deg$(x\sb {3} )\ge 4$, and deg$(x\sb {2} )=$ deg$(x\sb {4} )= 3$. Then one of the following holds; [1]both of $\beta \sb {1}, \beta \sb {2}$ have degree $\ge 9$; [2]one of $\beta \sb {1}, \beta \sb {2}$ has degree $\ge 2519$. Therefore the sum of the charge sent from these two faces is at least $\frac{1}{2}$. The two faces $\beta \sb {3}, \beta \sb {4}$ similarly send to $\alpha $ the charge at least $\frac{1}{2}$. Hence $c\sp {*}(\alpha ) \ge c(\alpha ) + ( \frac{1}{2} \times 2) + (\frac{1}{2}  \times 2 )= 0$.

[e] We assume that three vertices of $x\sb {1}, x\sb {2}, x\sb {3}, x\sb {4}$ are of degree $\ge 4$. Let deg$(x\sb {1} )\ge 4$, deg$(x\sb {2} )\ge 4$, deg$(x\sb {3} )\ge 4$, and deg$(x\sb {4} )= 3$. Then [1] both of $\beta \sb {3}, \beta \sb {4}$ have degree $\ge 9$, or [2] one of them has degree $\ge 2519$. Therefore $c\sp {*}(\alpha ) \ge c(\alpha ) + ( \frac{1}{2} \times 3) + \frac{1}{2} = 0$.

[f] We finally assume that all of $x\sb {1}, x\sb {2}, x\sb {3}, x\sb {4}$ are of degree $\ge 4$. Then $c\sp {*}(\alpha ) \ge c(\alpha ) +  \frac{1}{2} \times 4 = 0$.\\

{\it 2-3}: $k=5$. Let $x\sb {1},\ldots , x\sb {5}$ be the vertices of $\alpha $, and $\beta \sb {1},\ldots ,\beta \sb {5}$ the faces, similarly as in the previous cases.

[a] We assume that all of $x\sb {1},\ldots , x\sb {5}$ are of degree $3$. Then at most two of $\beta \sb {1},\ldots ,\beta \sb {5}$ are $5$-, $6$-faces. If two of them, say $\beta \sb {1}, \beta \sb {3}$, are $5$-, $6$-faces, then the other three faces have degree $\ge 2519$, and then $c\sp {*}(\alpha ) \ge c(\alpha ) +  \frac{2}{5} \times 3 \ge 0$. If exactly one of them, say $\beta \sb {1}$, is $5$-, $6$-faces, then its neighboring faces $\beta \sb {2}$, $\beta \sb {5}$ are of degree $\ge 2519$ and the other two faces $\beta \sb {3}$, $\beta \sb {4}$ are of degree $\ge 7$, and then $c\sp {*}(\alpha ) \ge c(\alpha ) +  \frac{2}{5} \times 2 + \frac{1}{5} \times 2 \ge 0$. Hence all five faces are of degree $\ge 7$, and then $c\sp {*}(\alpha ) \ge c(\alpha ) +  \frac{1}{5} \times 5 = 0$.

[b] We assume that one of the five vertices, say $x\sb {1}$, has degree $\ge 4$. If deg$(\beta \sb {3})=5,6$, then its neighboring faces $\beta \sb {2}$, $\beta \sb {4}$ have degree $\ge 2519$, and then $c\sp {*}(\alpha ) \ge c(\alpha ) + \frac{1}{5} + \frac{2}{5} \times 2 = 0$. Thus deg$(\beta \sb {3})\ge 7$. One of the followings holds; [1] both of $\beta \sb {4}, \beta \sb {5}$ have degree $\ge 7$; or [2] one of $\beta \sb {4}, \beta \sb {5}$ has degree $\ge 2519$. The sum of the charge sent from these two faces is at least $\frac{1}{5}$ in either case. If deg$(\beta \sb {2})=5,6$, then $\beta \sb {1}$ and $\beta \sb {3}$ have degree $\ge 2519$, and then $c\sp {*}(\alpha ) \ge c(\alpha ) +  \frac{1}{5} + (\frac{1}{5} + \frac{2}{5} + \frac{1}{5} ) = 0$. Thus deg$(\beta \sb {2})\ge 7$. If deg$(\beta \sb {1})=5,6$, then $\beta \sb {2}$ has degree $\ge 2519$, and then $c\sp {*}(\alpha ) \ge c(\alpha ) + \frac{1}{5} + (\frac{2}{5} + \frac{1}{5} + \frac{1}{5} ) = 0$. Thus deg$(\beta \sb {1})\ge 7$. We can similarly deduce that deg$(\beta \sb {4})\ge 7$, deg$(\beta \sb {5})\ge 7$, therefore $c\sp {*}(\alpha ) \ge c(\alpha ) + \frac{1}{5} + (\frac{1}{5} \times 3 + \frac{1}{10} \times 2 ) = 0$. 

[c] We assume that consecutive two vertices of $x\sb {1},\ldots , x\sb {5}$ are of degree $\ge 4$. Let deg$(x\sb {1} )\ge 4$, deg$(x\sb {2} )\ge 4$, and deg$(x\sb {3} )=$ deg$(x\sb {4} )=$ deg$(x\sb {5} )= 3$. If deg$(\beta \sb {3})=5,6$, then $\beta \sb {2} $ and $\beta \sb {4} $ are of degree $\ge 2519$, and then $c\sp {*}(\alpha ) \ge 0$. Therefore deg$(\beta \sb {3} )\ge 7$, and similarly deg$(\beta \sb {4} )\ge 7$. If deg$(\beta \sb {2})=5,6$, then $\beta \sb {3} $ is of degree $\ge 2519$, and then $c\sp {*}(\alpha ) \ge 0$. Therefore deg$(\beta \sb {2})\ge 7$ and deg$(\beta \sb {5})\ge 7$ are similarly deduced. And then $c\sp {*}(\alpha ) \ge c(\alpha ) + \frac{1}{5} \times 2 + (\frac{1}{5} \times 2 + \frac{1}{10} \times 2 ) = 0$. The other case that non-consecutive two vertices of $x\sb {1},\ldots , x\sb {5}$ are of degree $\ge 4$ is similar.

[d] We can treat the remaining cases that three, four or five vertices of $x\sb {1},\ldots , x\sb {5}$ are of degree $\ge 4$ as well as above cases, and can deduce that $c\sp {*}(\alpha ) \ge 0$.\\

{\it 2-4}: $k=6$. Then $c\sp {*}(\alpha )= c(\alpha ) = 0$.\\

{\it 2-5}: $k=7,8$. The transfer from $\alpha $ is possible along $(3,k)$-, $(4,k)$-, $(5,k)$-edges which are weak or semi-weak. Since there are no consecutive two such edges which are weak, and since there are no consecutive three such edges which are semi-weak, $c\sp {*}(\alpha ) \ge c(\alpha ) - \frac{1}{5} \times 3= \frac{2}{5}$ if $k=7$, and $c\sp {*}(\alpha ) \ge c(\alpha ) - \frac{1}{5} \times 4= \frac{6}{5}$ if $k=8$.\\

{\it 2-6}: $9\le k \le 12$. we can similarly deduce that $c\sp {*}(\alpha ) \ge c(\alpha ) -  \frac{1}{2} \times \lfloor \frac{k}{2} \rfloor $. This value is $1, \frac{3}{2}, \frac{5}{2}, 3$ for $k = 9,10,11,12$, respectively. \\

{\it 2-7}: $13\le k \le 2518$. Then $c\sp {*}(\alpha ) \ge c(\alpha ) -  1 \times \lfloor \frac{k}{2} \rfloor = (k-6)- \lfloor \frac{k}{2} \rfloor $, and this value is not less than $ \frac{1}{21}k $ when $13\le k \le 2518$.\\

{\it 2-8}: $k \ge 2519$. The transfer from $\alpha $ is possible along $(3,k)$-, $(4,k)$-, $(5,k)$-edges which are weak or semi-weak. We notice that there are no consecutive two (resp. consecutive three) $(3,k)$-edges which are weak (resp. semi-weak), and that there are no consecutive $(4,k)$-edges which are weak. Therefore the transfer from $\alpha $ is maximized when [1] every other edges on $\alpha $ are weak $(3,k)$-edges and $k$ is even; [2] except two consecutive $(3,k)$-edges which are incident with a $(3,3,4,k)$-vertex on $\alpha $, every other edges on $\alpha $ are weak $(3,k)$-edges and $k$ is odd. If $k$ is even, $c\sp {*}(\alpha ) \ge c(\alpha ) - \frac{19}{10} \times \frac{k}{2}= \frac{1}{20}k-6 \ge \frac{1}{21}k $ because $k \ge 2520$. If $k$ is odd, $c\sp {*}(\alpha ) \ge c(\alpha ) - ( \frac{19}{10} \times \frac{k-3}{2} + 2 + \frac{1}{2})  = \frac{1}{20}k-6 + \frac{7}{20} > \frac{1}{21}k $ because $k \ge 2519$.
\end{case}
As a consequence, we establish Lemma \ref{th:b}.
\end{proof}
\ \\
\textbf{Rule B}.\\
Each major faces $\alpha $ sends its charge $c\sp {*}(\alpha )$ equally to its incident vertices.\\

We notice that each major face sends at least $\frac{1}{21}$ charge to its incident vertices because $\min \{ \frac{2}{5} \times \frac{1}{7}; \frac{6}{5}\times \frac{1}{8}; 1\times \frac{1}{9}; \frac{3}{2}\times \frac{1}{10}; \frac{5}{2}\times \frac{1}{11}; 3\times \frac{1}{12}; \frac{1}{21} k \times \frac{1}{k} \} = \frac{1}{21} $.

\begin{lemma}
After applying Rule \textbf{B}, we have the new charge $c\sp {**}(v)$:
\begin{eqnarray*}
c\sp {**}(v) \ge \frac{1}{21} \text{ for all vertices } v\in V(G).
\end{eqnarray*}
\label{th:c}
\end{lemma}
\begin{proof}
Let $v$ be a $k$-vertex, $k \ge 3$.

We first assume that $k=3$. Since $G$ is a counterexample, this vertex $v$ corresponds to one of the following types: $(3,i, \ge 2519)$; $i= 6, \ldots , 12$, $(3,\ge 13, \ge 13)$, $(4,j, \ge 2519)$; $j =5, \ldots , 8$, $(4,\ge 9, \ge 9)$, $(5,5, \ge 2519)$, $(5,6, \ge 2519)$, $(5,\ge 7, \ge 7)$, $(\ge 6,\ge 6, \ge 7)$. In each case, $v$ is incident with at least one major face, therefore $c\sp {**}(v) \ge \frac{1}{21} $.\\

We next assume that $k=4$. This vertex is one of the following types: $(3,3,4, \ge 2519)$; $(3,3,5, \ge 2519)$; $(3,3,\ge 6, \ge 7)$; $(3,4,4, \ge 7)$; $(3,\ge 4, \ge 5, \ge 5)$; $(\ge 4,\ge 4,\ge 4,\ge 5)$. If $v$ has the type $(3,i,6, 6)$ as in Rule \textbf{A$2$}, then $c\sp {**}(v) \ge c(v) -(1+\frac{1}{10} +\frac{1}{2} ) = \frac{2}{5} \ge \frac{1}{21}$. If $v$ has the type $(3,\ge 4, \ge 5, \ge 5)$, then the charge of $v$ is still remained, i.e., $c\sp {**}(v) \ge c\sp {*}(v) \ge c(v) -(1+\frac{1}{2}+\frac{1}{5} +\frac{1}{5}  ) = \frac{1}{10} \ge \frac{1}{21}$. If $v$ has the type $(\ge 4,\ge 4,\ge 4,\ge 5)$, then similarly $c\sp {**}(v) \ge c(v) -(\frac{1}{2}+\frac{1}{2}+\frac{1}{2} +\frac{1}{5}  ) = \frac{3}{10} \ge \frac{1}{21}$. For the other cases, $v$ is incident with at least one major face, and $c\sp {**}(v) \ge \frac{1}{21} $.\\

We set $k=5$. This vertex is one of the following types: $(3,3,3,3, \ge 7)$, $(3,3,3,\ge 4,\ge 5)$, $(\ge 3,\ge 3,\ge 4,\ge 4,\ge 4)$. If $v$ is a $(3,3,3,3, \ge 7)$-vertex, then $c\sp {**}(v) \ge c\sp {*}(v) + \frac{1}{21} \times 1 =  \frac{1}{21} $. If $v$ is a $(3,3,3,6,6)$-vertex, then $c\sp {**}(v) = c(v) -(1+ \frac{1}{10})\times 3 =4- \frac{33}{10} \ge \frac{1}{21} $. For the other cases, we observe that the charge of $v$ is still remained, and $c\sp {**}(v) \ge \frac{1}{21}$.\\

We set $k\ge 6$. This vertex $v$ has the type $(\ge 3,\ge 3,\ge 3,\ge 3,\ge 3,\ge 4)$ if $k=6$. In any case, the charge of $v$ is still remained after applying Rule \textbf{A$1-4$}, and we can deduce that $c\sp {**}(v) \ge \frac{1}{21}$.
\end{proof}
Euler's formula together with Lemma \ref{th:c} and the hypothesis $n > 126| \chi (\M) |$ yields
\begin{eqnarray*}
6|\chi (\M)| =\sum_{v\in V(G)}^{}c(v) + \sum_{\alpha \in F(G)}^{}c(\alpha ) =\sum_{v\in V(G)}^{}c\sp {**}(v) + \sum_{\alpha \in F(G)}^{}c\sp {**}(\alpha ) \\
 \ge \sum_{v\in V(G)}^{}c\sp {**}(v) \ge \frac{1}{21}n > 6|\chi (\M)|,
\end{eqnarray*}
a contradiction. This completes the proof of Theorem $2$.
\\
\\
\\
\section{Path Transferability of Graphs on Surfaces}
In this section we treat path transferability of graphs on general surfaces. We first prepare several notations: A {\it path} consists of distinct vertices $v\sb {0},v\sb {1}, \ldots ,v\sb {n}$ and edges $v\sb {0}v\sb {1}, v\sb {1}v\sb {2},  \ldots , v\sb {n-1}v\sb {n}$. Through this paper we assume that each path has a direction. The reverse path of $P$ is denoted by $P^{-1}$. The number of edges in a path $P$ is called its {\it length}, and a path of length $n$ is called an {\it $n$-path}. The last(resp. first) vertex of a path $P$ in its direction is called the {\it head} (resp. {\it tail} ) of $P$ and is denoted by $h(P)$(resp. $t(P)$); for $P= \langle v\sb {0}v\sb {1}\cdots v\sb {n-1}v\sb {n}\rangle$, we set $h(P)=v\sb {n}$ and $t(P)=v\sb {0}$. The set of all inner vertices of $P$, the vertices that are neither the head nor the tail, is denoted by $Inn(P)$.

We are interested in the movement of a path along a graph, which seems as a ``train'' moving on the graph: Let $P$ be an $n$-path. If $h(P)$ has a neighboring vertex $v\notin Inn(P)$, then we have a new $n$-path $P\sp {\prime}$ by removing the vertex $t(P)$ from $P$ and adding $v$ to $P$ as its new head. We say that $P$ take a {\it step} to $v$, and denote it by $P\xrightarrow[]{v} P\sp {\prime}$ (or briefly $P\xrightarrow[]{} P\sp {\prime}$). If there is a sequence of $n$-paths $P\xrightarrow[]{} \cdots \xrightarrow[]{} Q$, then we say that {\it $P$ can transfer} ({\it or move}) {\it to $Q$}, and denote it by $P\dashrightarrow Q$.
 A graph $G$ is called {\it $n$-path-transferable} or {\it $n$-transferable} if $G$ has at least one $n$-path and if $P \dashrightarrow Q$ for any pair of directed $n$-paths $P, Q$ in $G$. The maximum number $n$ for which $G$ is $n$-path-transferable is called the {\it path transferability} of $G$. The following result says that for a fixed surface the number of polyhedral maps whose path transferability are more than $12$ is finite:\\
\\
\textbf{Main Theorem.}
If a polyhedral map $G$ on a surface $\M$ of Euler characteristic $\chi (\M) \le 0$ has more than $126|\chi (\M)|$ vertices, then path transferability of $G$ is at most $12$.
\begin{proof}
Let $G$ be a graph as above. By Theorem $2$, $G$ contains one of the light vertices in Table \ref{tab:1}. We assume that such a vertex has the type $(3,12,\le 2518)$. Then we can find a path of length $13$ which cannot move any longer (see Fig.$1$), therefore path transferability of $G$ is at most $12$. For the other types, we can similarly find clogged paths of length $\le 12$ around the light vertices. Hence path transferability of $G$ is at most $12$.
\end{proof}

\begin{figure}[htbp]
\begin{center}
\includegraphics[width=0.40\linewidth]{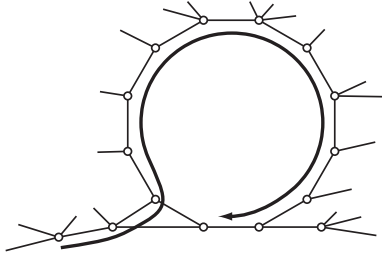}
\end{center}
\caption{A path clogged around a light vertex $(3,12,\le 2518)$.}
\label{fig:1}
\end{figure}

Let $G$ be a polyhedral map on a surface whose faces are of degree $6$, and $G^ {\Delta}$ the truncated graph of $G$. This graph $G^ {\Delta}$ has path transferability $12$, thus the value $12$ in this theorem is best possible (see Fig.$2$).

\begin{figure}[htbp]
\begin{center}
\includegraphics[width=0.40\linewidth]{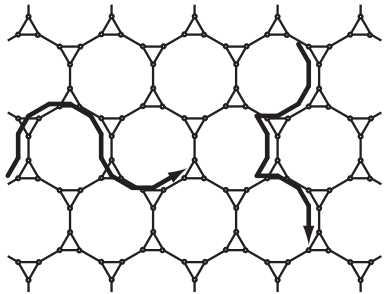}
\end{center}
\caption{Paths of length twelve can move from one to another in this graph.}
\label{fig:2}
\end{figure}

\section{Combinatorial Curvature}
\begin{figure}[htbp]
\begin{center}
\psfrag{a}{$(3,3, \forall )$}
\psfrag{b}{$(3,4, \forall )$}
\psfrag{c}{$(3,5, \forall )$}
\psfrag{d}{$(4,4, \forall )$}
\psfrag{e}{$(6,6,6)$}
\psfrag{f}{$(3,3,3, \forall )$}
\psfrag{g}{$(3,3,6,6)$}
\psfrag{h}{$(3,4,4,6)$}
\psfrag{i}{$(4,4,4,4)$}
\psfrag{j}{$(3,3,3,3,6)$}
\psfrag{k}{$(3,3,3,4,4)$}
\psfrag{l}{$(3,3,3,3,3,3)$}
\includegraphics[width=0.95\linewidth]{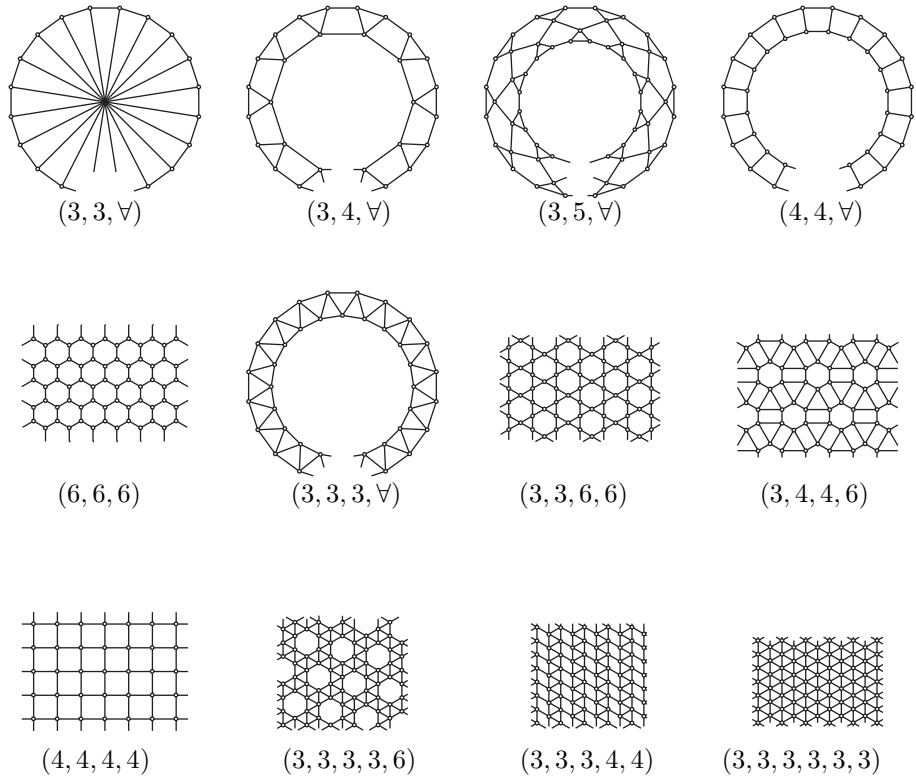}
\end{center}
\caption{Light vertices which cannot remove from Table \ref{tab:1}.}
\label{fig:3}
\end{figure}
By considering the graphs in Fig.$3$, we can see that the several types in Table \ref{tab:1} is in some sense tight; we cannot remove the types $(a\sb {1}, a\sb {2}, \ldots , a\sb {n})$ with $\dagger $ mark from the list. On the other hand, Higuchi \cite {Hi} studied the {\it combinatorial curvature}, introduced by Gromov \cite {Gr}, that is defined as 
\begin{eqnarray*}
\Phi (v)=1 - \frac{\text{deg}(v)}{2} + \sum_{\alpha \in F(v)}^{} \frac{1}{\text{deg}(\alpha )} ,
\end{eqnarray*}
where $F(v)$ is the set of faces incident with a vertex $v$.
Higuchi conjectured the following:
\begin{conjecture}[Higuchi]
Let $G$ be a finite or infinite planar graph. If $\Phi (v) >0 $ for all $v \in V(G)$, then $G$ has finite number of vertices.
\end{conjecture}

This conjecture was partly confirmed by Higuchi himself \cite {Hi} for some special cases, and by Sun and Yu \cite {SuYu} for the case of $3$-regular graphs. The conjecture is fully solved by Devos and Mohar \cite {DeMo} by establishing a Gauss-Bonnet inequality on polygonal surface. B.~Chen and G.~Chen \cite {ChCh} further investigated this study.

We will expect the following for a polyhedral map on a fix surface:
\begin{conjecture}
Let $G$ be a simple polyhedral map on a surface $\M$ of Euler characteristic $\chi (\M) \le 0$. There exists a constant number $c\sb{\M}$ for each $\M$ such that $G$ contains a vertex $v$ with $\Phi (v) \ge 0 $ if $ |V(G)| > c\sb{\M}|\chi (\M)|$.
\end{conjecture}

This means that there exists a vertex whose surrounding area looks convex or flat if a polyhedral map has sufficiently large number of vertices. From the aspect of the combinatorial curvature, the upper bound $2518$ in Table \ref{tab:1} will be expected to improve as $(3,7,\le 42)$, $(3,8,\le 24)$, $(3,9,\le 18 )$, $(3,10,\le 15)$, $(3,11,\le 13)$, $(3,12,12)$, $(4,5,\le 20)$, $(4,6,\le 12)$, $(4,7,\le 9)$, $(4,8,8)$, $(5,5,\le 10)$, $(5,6,\le 7 )$, $(3,3,4,\le 12)$, $(3,3,5,\le 7 )$, respectively, for some number $c\sb{\M}$.


\end{document}